\documentclass[12pt, oneside]{amsart}
\usepackage[all]{xy}
\usepackage{amsfonts}
\usepackage{amsmath,amssymb, amsthm,amsxtra}
\usepackage{amscd} 
\usepackage{verbatim}  
\usepackage{xspace}
\usepackage{url}
\usepackage{graphicx}
\usepackage{color}

\newcommand{\rk} {\operatorname{rk}}

\newcommand{\coeff} {\operatorname{coeff}}
\newcommand{\const} {\operatorname{const}}
\newcommand{\sgn} {\operatorname{sgn}}
\newcommand{\ch} {\operatorname{ch}}

\newcommand{\grassmann}{{\mathbb G_X({d}, \mathcal E)}}

\newcommand{\Gr}{{G}}
\newcommand{\X}[1]{{\mathbb F_X^{#1}(\mathcal E)}}
\newcommand{\G}[1]{{\mathbb F_{\Gr}^{#1}(\mathcal Q)}}

\newcommand{\HH}[1]{{ A^*{(#1)}}} 

\newcommand{\pr}[1]{{ \{ {#1} \}!  }} %

\newcommand{\ronbuntitle}{%
Degree Formulae for Grassmann Bundles, II%
}

\title[\ronbuntitle]{\ronbuntitle}

\author{Hajime KAJI$^*$ and Tomohide TERASOMA$^{**}$}

\address{%
Department of Mathematics, 
School of Science and Engineering, 
Waseda University 
\newline 
\indent 
3-4-1 Ohkubo, Shinjuku, Tokyo 169--8555, JAPAN }
\email{kaji@waseda.jp}

\address{%
Department of Mathematical Science,
University of Tokyo
\newline 
\indent 
3-8-1 Komaba, Meguro, Tokyo 153-8914, JAPAN}
\email{terasoma@ms.u-tokyo.ac.jp}

\thanks{
\noindent $^*$
Department of Mathematics, 
School of Science and Engineering, 
Waseda University.  
\newline 
\indent 
3-4-1 Ohkubo, Shinjuku, Tokyo 169--8555, JAPAN. 
\newline 
\indent 
{\it E-mail address:} {\tt kaji@waseda.jp}.
\\
\indent 
$^{**}$
Department of Mathematical Science,
University of Tokyo. 
\newline 
\indent 
3-8-1 Komaba, Meguro, Tokyo 153-8914, JAPAN. 
\newline 
\indent 
{\it E-mail address:} {\tt terasoma@ms.u-tokyo.ac.jp}%
} 


\theoremstyle{plain}
	\newtheorem{theorem}{Theorem}[section]
	\newtheorem*{theorem*}{Theorem}
	\newtheorem{corollary}[theorem]{Corollary}
	\newtheorem{proposition}[theorem]{Proposition}
	\newtheorem{lemma}[theorem]{Lemma}

\theoremstyle{definition}
	\newtheorem{definition}[theorem]{Definition}

\theoremstyle{remark}
	\newtheorem{remark}[theorem]{Remark}

\numberwithin{equation}{section}

\begin{document}

\begin{abstract} 
Let 
$X$ be a non-singular quasi-projective variety over a field,  
and let $\mathcal E$ 
be a vector bundle over $X$. 
Let $\grassmann$ be 
the Grassmann bundle of $\mathcal E$ over $X$ 
parametrizing corank $d$ subbundles of $\mathcal E$
with projection $\pi : \grassmann \to X$, 
and let 
$ \mathcal Q \gets \pi^*\mathcal E$ 
be the universal quotient bundle of rank $d$. 
In this article, a closed formula for 
$\pi_{*}\ch (\det \mathcal Q)$, 
the push-forward of 
the Chern character of 
the Pl\"ucker line bundle 
$\det \mathcal Q$ by $\pi$ is 
given in terms of the Segre classes of $\mathcal E$. 
Our formula yields a degree formula for 
 $\grassmann$ 
with respect to $\det \mathcal Q$ 
when $X$ is projective and $\wedge ^d \mathcal E$ is very ample. 
To prove the formula above, 
a push-forward formula 
in the Chow rings 
from a partial flag bundle of $\mathcal E$ to $X$ 
is 
given. 
\end{abstract}

\setcounter{section}{-1}
\maketitle
\section{Introduction}

Let $X$ be a non-singular quasi-projective variety of dimension $n$ 
defined over a field of arbitrary characteristic, 
and let $\mathcal E$ be a vector bundle of rank $r$ over $X$. 
Let $\grassmann$ be 
the Grassmann bundle of $\mathcal E$ over $X$ 
parametrizing corank $d$ subbundles of $\mathcal E$ 
with projection $\pi : \grassmann \to X$,  
and 
let $\mathcal Q \gets \pi^*\mathcal E$ 
be the universal quotient bundle of rank ${d}$ on $\grassmann$.
We denote by $\theta$ 
the first Chern class 
$c_1(\det\mathcal Q)= c_1(\mathcal Q)$
of $\mathcal Q$, and call $\theta$ the 
 {\it Pl\"ucker class} of $\grassmann$: 
 In fact, the determinant bundle $\det \mathcal Q$ is 
isomorphic to the pull-back of the tautological line bundle 
$\mathcal O_{\mathbb P_X(\wedge^{d} \mathcal E)}(1)$ 
of $\mathbb P_X(\wedge^{d} \mathcal E)$ by the relative
Pl\"ucker embedding over $X$.

The purpose of this article is to study 
the push-forward  of 
powers of the Pl\"ucker class  to $X$ by $\pi$, 
namely, $\pi_*(\theta^N)$, 
where $\pi_{*} : A^{*+d( r - d )}(\grassmann)\to A^{*}(X)$ is the push-forward by $\pi$ between 
the Chow rings. 
The main result is 
a closed formula for the push-forward of 
$\ch (\det\mathcal Q) :=\exp \theta = \sum_{N\ge 0} \frac{1}{N!}\theta^{N}$, 
the Chern character of $\det \mathcal Q$ 
in terms of the Segre classes of $\mathcal E$, as follows:

\begin{theorem}\label{theorem:main_theorem}
We have 
\begin{equation*} 
\pi_* \ch (\det \mathcal Q)
=
\sum_{k}  
\frac{  \prod_{0 \le i < j \le d-1}(k_i - k_j -i + j)}
{\prod_{0 \le i \le d-1} (r  + k_i   -i )!}
\prod_{0 \le i \le d-1}
s_{k_i}(\mathcal E) 
\end{equation*}
in 
$A^{*}(X)\otimes \mathbb Q$, 
where 
$k = (k_0, \dots , k_{d-1}) \in \mathbb Z_{\ge 0}^d$, 
and $s_i(\mathcal E)$ is the $i$-th Segre class of $\mathcal E$. 
\end{theorem}

The Segre classes $s_i(\mathcal E)$ here are
the ones satisfying
$s(\mathcal E, t)c(\mathcal E, -t)=1$ as in 
\cite{fujita}, \cite{laksov}, \cite{laksov-thorup}, 
where 
$s(\mathcal E, t)$ and $c(\mathcal E, t)$ are respectively 
the Segre series and the Chern polynomial of $\mathcal E$ in $t$.  
Note that our Segre class $s_i(\mathcal E)$ 
differs by the sign $(-1)^i$ 
from the one in \cite{fulton}.

Theorem \ref{theorem:main_theorem} yields

\begin{corollary}[Degree Formula for Grassmann Bundles]%
\label{corollary:degree_formula}
If $X$ is projective 
and 
$\wedge^{d} \mathcal E$ is very ample, then  $\grassmann$  
is embedded in the projective space $\mathbb P(H^0(X, \wedge^{d} \mathcal E))$ 
by the tautological line bundle $\mathcal O_{\grassmann}(1)$,  
and its degree is given by 
$$
\deg \grassmann 
={(d(r-d)+n)!}
\sum_{\vert k\vert = n}  
\frac{  \prod_{0 \le i < j \le d-1}(k_i - k_j -i + j)}
{\prod_{0 \le i \le d-1} (r  + k_i   -i )!}
\int_{X} 
\prod_{0 \le i \le d-1}
s_{k_i}(\mathcal E) , 
$$
where $\vert k \vert := \sum_i k_i$. 
\end{corollary}

Here a vector bundle $\mathcal F$ over $X$ is said to be {\it very ample} if 
the tautological line bundle $\mathcal O_{\mathbb P_X(\mathcal F)}(1)$ 
of $\mathbb P_X(\mathcal F)$ is very ample.

We also give a proof for the following:

\begin{theorem}
[\cite{kaji-terasoma}, \cite{manivel}]
\label{theorem:another_formula} 
We have
$$
\pi_*\ch (\det \mathcal Q) 
= \sum_{\lambda} 
\frac{1}{\vert \lambda +\varepsilon \vert !} f^{\lambda+\varepsilon} \varDelta_{\lambda} (s(\mathcal E))
$$
in 
$A^{*}(X)\otimes \mathbb Q$, 
where 
$\lambda =(\lambda_1 , \dots, \lambda_d)$ is a partition with 
$\vert \lambda \vert := \sum_i \lambda _i$, 
$\varepsilon := (r-d)^d =(r-d,\dots , r-d)$,
$f^{\lambda+\varepsilon}$ is the number of standard Young tableaux with shape $\lambda+\varepsilon$, 
and 
$\varDelta_{\lambda}(s(\mathcal E)):= \det[s_{\lambda_i+j-i}(\mathcal E)]_{1 \le i,j\le d}$ is 
the Schur polynomial 
in the Segre classes of $\mathcal E$ corresponding to $\lambda$. 
\end{theorem}

Note that 
our proofs for Theorem \ref{theorem:another_formula}
as well as Theorem \ref{theorem:main_theorem} 
do not use the push-forward formula of J\'ozefiak-Lascoux-Pragacz \cite{jlp}, 
while the proofs given in \cite{kaji-terasoma}, \cite{manivel} do. 
 We establish instead a new push-forward formula, as follows: 
Let $\X{d}$ be 
the partial flag bundle of $\mathcal E$ on $X$, 
parametrizing flags of subbundles of corank $1$ up to $d$ in $\mathcal E$, 
let $p : \X{d}\to X$ be the projection, 
and denote by $p_{*}: A^{*+c}(\X{d}) \to A^{*}(X)$ the push-forward by $p$, 
where $c$ is the relative dimension of $\X{d}/X$. 
Let $\xi_{0}, \dots , \xi_{d-1}$ be the set of Chern roots of $\mathcal Q$.  
It turns out (see \S1) that 
one may consider 
$A^{*+c}(\X{d})$ as an $A^{*}(X)$-algebra generated by the $\xi_{i}$. 
Then

\begin{theorem}[Push-Forward Formula]
\label{theorem:general_push_forward_formula}
For any polynomial $F \in A^{*}(X)[T_0 \dots, T_{d-1}]$, we have 
$$
p_{*} F(\underline{\xi})
=
\const_{\underline t} 
\Big(
\Delta(\underline t)
\prod_{i=0}^{d-1}
t_i^{r -d}
F(1/\underline{t})
\prod_{i=0}^{d-1}
s(\mathcal E, t_i)
\Big) 
,
$$
in 
$A^{*}(X)$, where 
$\underline{\xi}:=({\xi_0}, \dots, {\xi_{d-1}})$, 
$\const_{\underline{t}}(\cdots)$ denotes   
the constant term in the Laurent expansion 
of $\cdots$ in 
$\underline{t} := (t_0 , \dots, t_{d-1})$, 
$\Delta(\underline t):=\prod_{0 \le i < j \le d-1}(t_i - t_j)$
and 
$F(1/\underline t) := F(1/t_0, \dots, 1/t_{d-1})$. 
\end{theorem}

The contents of this article are organized as follows: 
The general theories 
\cite[\S6]{laksov-thorup}, \cite[\S\S0--1]{scott} 
on the structure of Chow ring 
of certain partial flag bundles 
are reviewed in \S1. 
Then, Theorem \ref{theorem:general_push_forward_formula} is proved in \S2, 
by which it is shown that 
$\pi_*\ch(\det Q)$ 
is given as the constant term of a certain Laurent series  
with coefficients in the Chow ring $\HH{X}$ 
of $X$, denoted by $P(\underline{t})$  
(Proposition \ref{prop:Laurent_series}). 
To evaluate the constant term of  $P(\underline{t})$, 
in \S3, 
a linear form 
on the Laurent polynomial ring,  
denoted by $\Phi$, 
is introduced  
(Definition \ref{definition:linear_form}), 
and an evaluation formula 
is proved 
(Proposition \ref{prop:evaluation_formula}): 
The evaluation formula 
is the key in the final step to prove 
Theorems 
\ref{theorem:main_theorem} and 
\ref{theorem:another_formula}. 
In \S4, 
a generalization of Cauchy determinant formula is given 
(Proposition \ref{prop:generalization_of_Cauchy_identity}). 
This yields another proof of a push-forward formula for monomials of the $\xi_{i}$ (Lemma \ref{lemma:monomial}). 

\section{Set-up}

Let $X$ be a non-singular quasi-projective variety of dimension $n$ defined over a field $k$, 
let $\mathcal E$ be a vector bundle of rank $r $ on $X$, 
and let $\varpi : \mathbb P(\mathcal E) \to X$ be the projection. 
Denote by  $\xi$ the first Chern class of the tautological line bundle 
$\mathcal O_{\mathbb P(\mathcal E)}(1)$, and 
define a polynomial $P_{\mathcal E} \in A^*(X)[T]$ associated to $\mathcal E$ by setting 
$$
P_{\mathcal E} (T):= T^r -  c_1(\mathcal E)T^{r-1} + \cdots + (-1)^r c_r(\mathcal E) ,
$$
where $A^*(X)$ is the Chow ring of $X$. 
Then, $P_{\mathcal E}(\xi)=0$ by definition of the Chern classes 
(\cite[Remark 3.2.4]{fulton}), and  
\begin{equation}\label{equation:Chow_ring_proj_space_bundle}
A^*(\mathbb P(\mathcal E)) 
=\bigoplus_{0 \le i \le r-1}  A^*(X) \xi^i 
\simeq A^*(X)[T]/(P_{\mathcal E}(T))
\end{equation}
(\cite[Theorem 3.3 (b); Example 8.3.4]{fulton}).
Let $\varpi_* : A^{*+r-1}(\mathbb P(\mathcal E)) \to A^*(X)$ 
be the push-forward by $\varpi$.  
Then $\varpi_* \alpha $ is equal to  
the coefficient of $\alpha$ in $\xi^{r-1}$,
denoted by 
$\coeff_{\xi}(\alpha)$, 
with respect to the decomposition \eqref{equation:Chow_ring_proj_space_bundle}
for $\alpha \in A^{*+r-1}(\mathbb P(\mathcal E))$ 
(\cite[Proposition 3.1]{fulton}): 
\begin{equation}\label{equation:push_forward_for_projective_space_bundle}
\varpi_* \alpha 
= 
\coeff_{\xi}(\alpha) 
\end{equation}

Denote by  $\X{d}$ 
the partial flag bundle of $\mathcal E$ on $X$, 
parametrizing flags of subbundles of corank $1$ up to $d$ in $\mathcal E$, 
and let $p : \X{d}\to X$ be the projection.  
Set $\mathcal E_0 := \mathcal E$, and 
let $\mathcal E_{i+1}$ be the kernel of the
canonical surjection from 
the pull-back of $\mathcal E_{i}$ to $\mathbb P(\mathcal E_{i})$, 
to the tautological line bundle  
$\mathcal O_{\mathbb P(\mathcal E_{i})}(1)$, 
with $\rk \mathcal E_{i} = r-i$ $(i \ge 0)$. 
Set $\xi_i := c_1(\mathcal O_{\mathbb P(\mathcal E_{i})}(1))$.  
We have an exact sequence on $\mathbb P(\mathcal E_{i})$, 
$$
0 
\to \mathcal E_{i+1}
\to \mathcal E_{i}
\to \mathcal O_{\mathbb P(\mathcal E_{i})}(1) 
\to 0, 
$$
and an equation of Chern polynomials, 
\begin{equation}\label{equation:Chern_polynomials}
c(\mathcal E_i,t)  = c(\mathcal E_{i+1},t) (1 + \xi_i t)  ,
\end{equation}
where we omit the symbol of the pull-back by the projection 
$\mathbb P_{\mathbb P(\mathcal E_{i})}(\mathcal E_{i+1}) \to \mathbb P(\mathcal E_{i})$. 
It is easily shown that 
the projection $p : \X{d} \to X$ 
decomposes as a 
successive composition of projective space bundles, 
$\mathbb P_{\mathbb P(\mathcal E_{i})}(\mathcal E_{i+1}) \to \mathbb P(\mathcal E_{i})$ 
$(i \ge 0)$:  
$$
p : \X{d} = \mathbb P(\mathcal E_{d-1}) 
\to  \mathbb P(\mathcal E_{d-2}) \to \cdots \to  
 \mathbb P(\mathcal E_{1}) 
\to \mathbb P(\mathcal E_{0}) \to X 
. 
$$
In fact, $\mathbb P(\mathcal E_i) \simeq \X{i+1}$ 
$(0\le i \le d-1)$. 
Using \eqref{equation:Chow_ring_proj_space_bundle} repeatedly, we see that 
the Chow ring of $\X{d}$ is given as follows:
\begin{equation}\label{equation:cohomology_ring}
\HH{\X{{d}}}
= 
\bigoplus_{\substack{0\le i_l \le {r} -l-1\\(0\le l \le {d} -1)}}
\HH{X} 
\xi_0^{i_0} 
\xi_1^{i_1} \cdots 
\xi_{{d}-1}^{i_{{d} -1}} 
= 
\frac{\HH{X}[T_0, T_1, \dots ,T_{{d} -1}]} 
{( \{ P_{\mathcal E_i}(T_i)\vert 0 \le i \le d-1 \} )} . 
\end{equation}
Denote by 
$p_* : A^{*+c}(\X{d})$ $\to A^{*}(X)$ 
the push-forward by $p$, 
where $c := \sum_{0 \le i \le d-1} (r-i-1)$,  
the relative dimension of $\X{d}/X$. 
Then, 
using \eqref{equation:push_forward_for_projective_space_bundle} repeatedly, we 
see that 
\begin{equation}\label{equation:trace}
p_*\alpha 
=
\coeff_{\underline{\xi}}(\alpha)
\end{equation}
for 
$\alpha \in \HH{\X{{d}}}$, 
where 
$\coeff_{\underline{\xi}}(\alpha)$
denotes the 
coefficient of $\alpha$ in $\xi_0^{r-1}\xi_1^{r-2}\cdots\xi_{d-1}^{r-d}$
with respect to 
the decomposition \eqref{equation:cohomology_ring}.

Let $\Gr:=\grassmann$ be the Grassmann bundle 
of corank $d$ subbundles of $\mathcal E$ on $X$,  
and let 
$\mathcal Q \gets \pi^*\mathcal E$ 
be the universal quotient bundle of rank ${d}$.  
Consider the flag bundle $\G{d-1}$ of $\mathcal Q$ on $\Gr$, 
parametrizing flags of subbundles of corank $1$ up to $d-1$ in $\mathcal Q$. 
Then, as in the case of $\X{d}$,  
the projection $\G{d-1}\to G$ 
decomposes as a 
successive composition of projective space bundles,  
$ \mathbb P_{\mathbb P(\mathcal Q_{i})}(\mathcal Q_{i+1})\to\mathbb P(\mathcal Q_{i})$ $(i \ge 0)$: 
$$
q : \G{d-1} = \mathbb P(\mathcal Q_{d-2}) 
\to  \mathbb P(\mathcal Q_{d-2}) \to \cdots \to  
 \mathbb P(\mathcal Q_{1}) 
\to \mathbb P(\mathcal Q_{0}) \to G 
, 
$$
where 
$\mathcal Q_0 := \mathcal Q$, and 
$\mathcal Q_{i+1}$ is the kernel of the
canonical surjection from 
the pull-back of $\mathcal Q_{i}$ to $\mathbb P(\mathcal Q_{i})$, 
to the tautological line bundle  
$\mathcal O_{\mathbb P(\mathcal Q_{i})}(1)$, 
with $\rk \mathcal Q_{i} = d-i$ $(i \ge 0)$: 
In fact, $\mathbb P(\mathcal Q_i) \simeq \G{i+1}$ 
$(0\le i \le d-2)$ and 
$\mathbb P_{\mathbb P(\mathcal Q_{d-2})}(\mathcal Q_{d-1}) 
\simeq \mathbb P(\mathcal Q_{d-2}) \simeq \G{d-1}= \G{d}$.  
It follows from the construction of the $\mathcal Q_i$ that 
the Pl\"ucker class $\theta:=c_1(\det\mathcal Q)=c_1(\mathcal Q)$ is 
equal to the sum of the first Chern classes  $c_1(\mathcal O_{\mathbb P(\mathcal Q_{i})}(1))$ 
$(0 \le i \le d-1)$ 
in $\HH{\G{{d}-1}}$, 
where $\mathcal O_{\mathbb P(\mathcal Q_{d-1})}(1) = \mathcal Q_{d-1}$
via 
$\mathbb P_{\mathbb P(\mathcal Q_{d-2})}(\mathcal Q_{d-1})\simeq \mathbb P(\mathcal Q_{d-2})$.

It follows 
from the construction of the $\mathcal E_i$ that 
$\mathcal E_d$ is a corank $d$ subbundle 
 of $p^*\mathcal E$ on $\X{d}$, 
which induces a morphism, 
$r : \X{d} \to G$ over $X$
by the universal property of the Grassmann bundle $G$. 
Then it turns out that 
$\G{d-1}$ is 
naturally isomorphic to $\X{d}$ over $G$ via $r$, 
as is easily verified by using the universal property of flag bundles: 
We identify them via the natural isomorphism  
 $\G{d-1} \simeq \X{d}$.  
Under this identification, it follows that $p=\pi \circ q$ and 
$\xi_i 
= c_1(\mathcal O_{\mathbb P(\mathcal E_{i})}(1)) 
= c_1(\mathcal O_{\mathbb P(\mathcal Q_{i})}(1)) 
$
in $\HH{\X{{d}}}=\HH{\G{{d}-1}}$
$(0 \le i \le d-1)$, 
where 
the symbol of pull-back to $\X{d}=\G{d-1}$ is omitted, as before. 
Thus we have  
\begin{equation}\label{equation:theta}
q^*\theta = \xi_0 + \cdots + \xi_{d-1} 
\end{equation}
in $\HH{\X{d}}=\HH{\G{d-1}}$. 
For details, we refer to \cite[\S6]{laksov-thorup}, \cite[\S\S0--1]{scott}.

\section{Laurent series}

We keep the same notation as in \S1.

\begin{lemma}\label{lemma:xi} 
For any non-negative integer ${p}$, 
$$
\coeff_{\xi}(\xi^{p}) = \const_t( t^{-{p}+r-1}s(\mathcal E,t)) , 
$$
where 
$\const_t(\cdots)$ denotes 
the constant term 
in the Laurent expansion of $\cdots$ in $t$. 
\end{lemma}

\begin{proof}
Set 
$R_{{p}}(x_{{p}} , \dots , x_{{{p}}-r}):=\sum_{i=0}^{r}(-1)^i c_i(\mathcal E) x_{{{p}}-i}$, and 
consider a recurring relation,
$R_{{p}}(x_{{p}} , \dots , x_{{{p}}-r})=0$ $({{p}} \ge r)$ 
for $\{ x_i \} \subseteq A^*(X)$. 
If $a_{{p}}:= \coeff_{\xi}(\xi^{{p}})$, then 
$$
R_{{p}}(a_{{p}} , \dots , a_{{{p}}-r})
=\coeff_{\xi}\Big(\sum_{i=0}^{r}(-1)^i c_i(\mathcal E)\xi^{{p}-i}\Big)
= 0
$$ 
by $P_{\mathcal E}(\xi)=0$. 
On the other hand, 
if $b_{{p}}:= \const_t( t^{-{{p}}-1+r}s(\mathcal E,t))$, then 
$$
R_{{p}}(b_{{p}} , \dots , b_{{{p}}-r})
=\const_t\Big(\sum_{i=0}^{r} c_i(\mathcal E)  (-t)^{i}t^{-{{p}}-1+r}s(\mathcal E,t)\Big)
=\const_t(t^{-{{p}}-1+r})= 0
$$ 
by 
$c(\mathcal E,-t)s(\mathcal E,t)=1$. 
Thus both of $\{a_{{p}}\}$ and $\{b_{{p}}\}$ satisfy the recurring relation $R_{{p}} = 0$, 
so that 
$a_{{p}} = b_{{p}}$ for all ${{p}}$: 
Indeed, 
$a_{r}=b_{r}=c_1(\mathcal E)$, 
$a_{r-1}=b_{r-1}=1$， 
and 
$a_{{p}} =b_{{p}} = 0$ if $0\le {{p}} \le r-2$.
 We here note that 
$x_{{p}}$ is 
determined by 
$x_{{{p}}-1}, \dots, x_{{{p}}-r}$ if $R_{{p}}(x_{{p}}, \dots x_{{{p}}-r})=0$. 
\end{proof}

\begin{lemma}\label{lemma:monomial}
For any non-negative integers $p_0, \dots , p_{d-1}$, we have 
$$
\coeff_{\underline{\xi}}(\xi_{0}^{p_{0}} \cdots \xi_{d-1}^{p_{d-1}} )
=
\const_{\underline t} 
\Big(
\Delta(\underline{t})
\prod_{i=0}^{d-1}
t_i^{-p_i + r -d} s(\mathcal E, t_i)
\Big) 
,
$$
where 
$\const_{\underline{t}}(\cdots)$ denotes   
the constant term in the Laurent expansion 
of $\cdots$ in 
$\underline{t} := (t_0 , \dots, t_{d-1})$, 
and 
$\Delta(\underline{t}):= 
\prod_{0 \le i < j \le d-1}(t_i - t_j)$ is
the Vandermonde polynomial of 
$\underline t$.
\end{lemma}

\begin{proof}
Since 
$s(\mathcal E_{d-1}, t_{d-1})= (1-\xi_{d-2}t_{d-1}) s(\mathcal E_{d-2}, t_{d-1}) $
by \eqref{equation:Chern_polynomials}, 
it follows from Lemma \ref{lemma:xi} that 
$$
\coeff_{{\xi_{d-1}}} (\xi_{d-1}^{p_{d-1}}) 
= 
\const_{t_{d-1}}(t_{d-1}^{-p_{d-1}+r -d} (1-\xi_{d-2}{t_{d-1}} )s(\mathcal E_{d-2}, t_{d-1}) )
$$
in $A^{*}(\mathbb P(\mathcal E_{d-2}))$,
where 
$\coeff_{{\xi_{d-1}}} (\cdots)$
denotes the coefficient of 
$\cdots$ in $\xi_{d-1}^{r -d}$. 
Therefore, using Lemma  \ref{lemma:xi} again, we have 
{\allowdisplaybreaks %
\begin{align*}
\coeff&_{{\xi_{d-2}}, {\xi_{d-1}}} 
(\xi_{d-2}^{p_{d-2}} \xi_{d-1}^{p_{d-1}})
\\=& 
\coeff_{{\xi_{d-2}}}
(
\xi_{d-2}^{p_{d-2}}
\const_{t_{d-1}}(t_{d-1}^{-p_{d-1}+r -d} (1-\xi_{d-2}{t_{d-1}} )s(\mathcal E_{d-2}, t_{d-1}) )
)
\\=& 
\coeff_{{\xi_{d-2}}}
(
\xi_{d-2}^{p_{d-2}} 
\const_{t_{d-1}}(t_{d-1}^{-p_{d-1}+r -d} s(\mathcal E_{d-2}, t_{d-1}) ) 
)
\\& 
\hskip 3pt
+
\coeff_{{\xi_{d-2}}}
(
\xi_{d-2}^{{p_{d-2}}+1}
\const_{t_{d-1}}(t_{d-1}^{-p_{d-1}+r -d} (-t_{d-1} )s(\mathcal E_{d-2}, t_{d-1}) )
)
\\=& 
\const_{t_{d-2}}(t_{d-2}^{-p_{d-2}+r -d+1} s(\mathcal E_{d-2}, t_{d-2}) ) 
\const_{t_{d-1}}(t_{d-1}^{-p_{d-1}+r -d} s(\mathcal E_{d-2}, t_{d-1}) ) 
\\& 
\hskip 3pt
+
\const_{t_{d-2}}(t_{d-2}^{-(p_{d-2}+1)+r -d+1} s(\mathcal E_{d-2}, t_{d-2}) ) 
\const_{t_{d-1}}(t_{d-1}^{-p_{d-1}+r -d} (-t_{d-1} )s(\mathcal E_{d-2}, t_{d-1}) )
\\=& 
\const_{t_{d-2}, t_{d-2}}
(t_{d-2}^{-p_{d-2}+r -d+1} s(\mathcal E_{d-2}, t_{d-2}) 
t_{d-1}^{-p_{d-1}+r -d} s(\mathcal E_{d-2}, t_{d-1}) ) 
\\& 
\hskip 3pt
+
\const_{t_{d-2}, t_{d-2}}
(t_{d-2}^{-p_{d-2}+r -d} s(\mathcal E_{d-2}, t_{d-2})
t_{d-1}^{-p_{d-1}+r -d} (-t_{d-1} )s(\mathcal E_{d-2}, t_{d-1}) )
\\=& 
\const_{t_{d-2}, t_{d-2}}
\Big(
(t_{d-2} - t_{d-1} )
\prod_{i=d-2}^{d-1}
t_{i}^{-p_{i}+r -d} 
s(\mathcal E_{d-2}, t_{i})
\Big) 
\end{align*}
}%
in $A^{*}(\mathbb P(\mathcal E_{d-3}))$, 
where 
$\coeff_{{\xi_{d-2}}, {\xi_{d-1}}} (\cdots)$  
denotes the coefficient of 
$\cdots$ in 
$\xi_{d-2}^{r-d+1}\xi_{d-1}^{r-d}$, 
and 
$\coeff_{{\xi_{d-2}}} (\cdots)$
the coefficient of 
$\cdots$ in $\xi_{d-2}^{r -d+1}$.  
Repeating this procedure, we obtain the conclusion.
\end{proof}

\begin{remark} 
Expanding the determinant $\Delta(\underline t)$ in the right-hand side 
in Lemma \ref{lemma:monomial}, using 
\eqref{equation:trace}, 
we obtain a formula, 
$p_*(\xi_{0}^{p_{0}}\cdots\xi_{d-1}^{p_{d-1}}) = \det [s_{p_i+j-r+1}(\mathcal E)]_{0 \le i , j \le d-1}$ 
in terms of the Schur polynomials in Segre classes of $\mathcal E$, 
which is equivalent to the 
determinantal formula 
\cite[8.1 Theorem]{laksov} with 
$f_i(\xi_i):= \xi_i^{p_i}$ $(0 \le i \le d-1)$. 
\end{remark}

\begin{proposition}\label{prop:general_push_forward_formula}
For any polynomial $F \in A^{*}(X)[T_0, \dots, T_{d-1}]$, we have 
$$
\coeff_{\underline {\xi}} (F(\underline{\xi}))
=
\const_{\underline t} 
\Big(
\Delta(\underline t)
\prod_{i=0}^{d-1}
t_i^{r -d}
F(1/\underline{t})
\prod_{i=0}^{d-1}
s(\mathcal E, t_i)
\Big) 
,
$$
where 
$\underline{\xi}:=({\xi_0}, \dots, {\xi_{d-1}})$, 
$\const_{\underline{t}}(\cdots)$ denotes   
the constant term in the Laurent expansion 
of $\cdots$ in 
$\underline{t} := (t_0 , \dots, t_{d-1})$, 
$\Delta(\underline t):=\prod_{0 \le i < j \le d-1}(t_i - t_j)$, 
and 
$F(1/\underline t) := F(1/t_0, \dots, 1/t_{d-1})$. 
\end{proposition}

\begin{proof}
This follows from 
Lemma \ref{lemma:monomial}. 
\end{proof}

\begin{proof}[Proof of Theorem \ref{theorem:general_push_forward_formula}] 
The assertion follows from 
\eqref{equation:trace} and 
Proposition \ref{prop:general_push_forward_formula}.
\end{proof}

\begin{proposition}\label{prop:Laurent_series}
With the same notation as in \S1, we have 
$$
\pi_* \ch(\det \mathcal Q)
= \const_{\underline{t}}(P(\underline{t})) 
, 
$$
where 
$\pi_* : A^{*+d( r - d )}(\grassmann) \otimes \mathbb Q 
\to A^{*}(X) \otimes \mathbb Q$ 
is the push-forward by $\pi$, 
$\ch(\det \mathcal Q)$ 
is the Chern character of $\det \mathcal Q$, 
$\const_{\underline{t}}(\cdots)$ denotes   
the constant term in the Laurent expansion 
of $\cdots$ in 
$\underline{t} := (t_0 , \dots, t_{d-1})$, 
and 
$$ 
P(\underline{t})
:=
\Delta(\underline t) 
\prod_{i=0}^{d-1}
t_i^{r-d-(d-1-i)}
\exp\Big(
\sum_{i=0}^{d-1}
\frac{1}{t_i} 
\Big) 
\prod_{i=0}^{d-1}
s(\mathcal E, t_i) 
. 
$$
\end{proposition}

Note that, though 
$\exp\big( \sum_{i=0}^{d-1} \frac{1}{t_i}  \big)$ is an element in 
$\mathbb Q[[\{ \frac{1}{t_i} \}_{0 \le i \le d-1}]]$, 
$\const_{\underline{t}}(P(\underline{t}))$ is well defined 
since the Segre series are polynomials in $\underline t$.

\begin{proof}
Since 
$\G{i+1} \to \G{i}$
is a $\mathbb P^{d-1-i}$-bundle, 
using \cite[Proposition 3.1]{fulton} repeatedly, for a non-negative integer $N$, 
we have 
$$
\theta^N= 
q_*(\xi_0^{d-1} \xi_1^{d-2}\cdots \xi_{d-2}q^*\theta^N)
,
$$
where 
$q$ is the composition of  the projections, 
$\G{d-1}\to \cdots \to \G{1} \to G$. 
It follows from 
\eqref{equation:theta} and the commutativity 
$p=\pi \circ q$ via the identification $\G{d-1} = \X{d}$ 
that 
{\allowdisplaybreaks %
\begin{align*}
\pi_*(\theta^N)
=&
\pi_* q_*( \xi_0^{d-1} \xi_1^{d-2}\cdots \xi_{d-2}q^*\theta^N)
\\ =& 
 \pi_*
q_*
\Big(
\prod_{i=0}^{d-1} \xi_i^{d-1-i}
\Big(
\sum_{i=0}^{d-1} \xi_i 
\Big)^N 
\Big)
= 
p_*\Big(
\prod_{i=0}^{d-1} \xi_i^{d-1-i}
\Big(
\sum_{i=0}^{d-1} \xi_i 
\Big)^N 
\Big)
, 
\end{align*}
}%
where 
$p$ is the composition of  the projections, 
$\X{d}\to \cdots \to \X{1} \to X$. 
Now, apply Theorem \ref{theorem:general_push_forward_formula} with 
$F:= 
\prod_{i=0}^{d-1} T_i^{d-1-i}
\Big(
\sum_{i=0}^{d-1} T_i 
\Big)^N $. 
Then, 
$$
p_*
\Big(
\prod_{i=0}^{d-1} \xi_i^{d-1-i}
\Big(
\sum_{i=0}^{d-1} \xi_i 
\Big)^N 
\Big)
=
\const_{\underline t} 
\Big(
\Delta(\underline t)
\prod_{i=0}^{d-1}
t_i^{r -d-(d-1-i)}
\Big( 
\sum_{i=0}^{d-1}
t_i^{-1} 
\Big) ^N 
\prod_{i=0}^{d-1}
s(\mathcal E, t_i)
\Big)
.
$$
Thus the conclusion follows with $\ch(\det \mathcal Q) = \exp (\theta)$. 
\end{proof}

\section{A linear form on the Laurent polynomial ring}

\begin{definition}\label{definition:linear_form}
Let $A$ be a $\mathbb Q$-algebra.  
We define a linear form 
$\Phi:A[\{t_i,\frac{1}{t_i}\}_{0 \le i \le d-1}]
\to A$ on the Laurent polynomial ring
$A[\{t_i,\frac{1}{t_i}\}_{0 \le i \le d-1}]$ 
by
\begin{equation*}
\Phi(f):=
\const_{\underline{t}}
\Big(
\Delta(\underline{t})
\exp\Big(
\sum_{i=0}^{d-1}
\frac{1}{t_i} 
\Big)
f (\underline t) \Big) 
\qquad 
\Big(f \in A\Big[\Big\{t_i,\frac{1}{t_i} \Big\}_{0 \le i \le d-1} \Big]\Big), 
\end{equation*}
where $\underline t:= (t_0, \dots,t_{d-1})$. 
\end{definition}

\begin{lemma}\label{lemma:constant_part_linear_form}
\begin{enumerate}
\item
\label{lemma:constant_part_linear_form_1}
Consider the natural action of 
the permutation group 
$\mathfrak S_d$ 
on 
%
\linebreak 
%
$A[\{t_i,\frac{1}{t_i}\}_{0 \le i \le d-1}]$
with $\sigma(t_i):= t_{\sigma(i)}\; (\sigma \in \mathfrak S_d)$. 
Then we have $\Phi(\sigma(f))=\sgn(\sigma)\Phi(f)$.
As a consequence, we have 
$$
\Phi\Big(\prod_{i=0}^{d-1}t_i^{-(d-1-i)}f(\underline{t})\Big)
=(-1)^{d(d-1)/2}\Phi\Big(\prod_{i=0}^{d-1}t_i^{-i}f(\underline{t})\Big)
$$
for a symmetric function $f(\underline{t})$.
\item
\label{lemma:constant_part_linear_form_3}
For a Schur polynomial $s_{\lambda}(\underline{t})$ and 
a symmetric function $f(\underline{t})$, we have 
$$
\Phi\Big(\prod_{i=0}^{d-1}t_i^{-i}f(\underline{t})s_{\lambda}(\underline{t})\Big)
=
\Phi\Big(\prod_{i=0}^{d-1} 
t_i^{-i+\lambda_{i+1}} 
f(\underline{t})
\Big)
. 
$$
\end{enumerate}
\end{lemma}

Here the {\it Schur polynomial} $s_{\lambda}(\underline t)$
in $\underline t=(t_0 , \dots, t_{d-1})$  
for a partition 
${\lambda} = (\lambda_1, \dots , \lambda_d)$ 
is the polynomial 
defined by 
$$
s_{\lambda}(\underline t)
:= 
\frac{\det[t_j^{\lambda_{i}+d-i}]
}
{\det [ t_{j} ^{d-i} ]
} 
= 
\frac{\det[t_j^{\lambda_{i}+d-i}]
}
{\Delta(\underline t) } 
, 
$$
where $1 \le i\le d$, $0 \le j \le d-1$
(see, {\it e.g.}, \cite[14.5 and A.9]{fulton}, \cite[Chapter I, \S3]{macdonald}).

\begin{proof}
\eqref{lemma:constant_part_linear_form_1}.  
The assertion 
is a direct consequence from the definition of $\Phi$ and a property of
 $\Delta(\underline{t})$.

\eqref{lemma:constant_part_linear_form_3}. 
Using \eqref{lemma:constant_part_linear_form_1}, we have 
{\allowdisplaybreaks %
\begin{align*} 
\Phi
\Big(
\prod_{i=0}^{d-1}
t_i^{-i}f(\underline{t})s_{\lambda}(\underline{t})
\Big)
&=
\frac{1}{d!} 
\Phi
\Big(
\prod_{i=0}^{d-1}t_i^{-(d-1)}
f(\underline{t})s_{\lambda}(\underline{t})
\sum_{\sigma \in \mathfrak{S}_d} 
\sgn(\sigma)
\prod_{i=0}^{d-1}t_{\sigma(i)}^{d-1-i}
\Big) 
\\
&=
\frac{1}{d!} 
\Phi
\Big(
\prod_{i=0}^{d-1}t_i^{-(d-1)}
f(\underline{t})s_{\lambda}(\underline{t})
\Delta(\underline{t})
\Big) 
\\
&=
\frac{1}{d!} 
\Phi
\Big(
\prod_{i=0}^{d-1}t_i^{-(d-1)}
f(\underline{t})
\det[t_j^{\lambda_l+d-l}]_{1 \le l \le d, 0 \le j \le d-1}
\Big) 
\\
&=
\frac{1}{d!} 
\sum_{\sigma \in \mathfrak{S}_d} 
\sgn(\sigma)
\Phi
\Big(
\prod_{i=0}^{d-1}
t_{\sigma(i)}^{-i+\lambda_{i+1}}
f(\underline{t})
\Big) 
= 
\Phi
\Big(
\prod_{i=0}^{d-1}t_i^{-i+\lambda_{i+1}} 
f(\underline{t})
\Big) 
.
\qedhere 
\end{align*} 
}%
\end{proof}

To simplify the notation, 
for a finite set of integers 
$\{ a_i  \}_{0 \le i \le d-1}$, 
set 
$$ 
\pr{a_i} 
:= \prod_{0 \le i \le d-1} a_i ! ,
\quad 
\Delta(a_i) := \prod_{0 \le i< j \le d-1}(a_i - a_j) 
.
$$ 
Setting $m! := \Gamma(m+1)$ for $m \in \mathbb Z$, 
we have $1/m! = 0$ if $m < 0$.

\begin{proposition}[Evaluation Formula] 
\label{prop:evaluation_formula}
For $k = (k_0 , \dots, k_{d-1}) \in \mathbb Z_{\ge 0}^d$,  
we have
$$
\Phi\Big(\prod_{i=0}^{d-1}t_i^{k_i}\Big)
=
\frac{(-1)^{d(d-1)/2} \Delta(k_i)}
{\{k_i+d-1\}!}.
$$
\end{proposition}

\begin{proof}
We have  
{\allowdisplaybreaks %
\begin{align*}
\Phi\Big(\prod_{i=0}^{d-1}t_i^{k_i}\Big)
&=
\const_{\underline{t}}
\Big(
\sum_{\sigma\in \mathfrak S_d}
\sgn(\sigma)\prod_{i=0}^{d-1}\Big(t_i^{k_i+d-1-\sigma(i)}
\exp
\Big(\frac{1}{t_i}\Big)
\Big)
\Big)
\\
&=
\sum_{\sigma\in \mathfrak S_d}
\sgn(\sigma)\prod_{i=0}^{d-1} 
\const_{{t_i}}
\Big(t_i^{k_i+d-1-\sigma(i)}
\exp
\Big(\frac{1}{t_i}\Big)
\Big)
\\
&=
\sum_{\sigma\in \mathfrak S_d}
\frac{\sgn(\sigma)}
{\{k_i+d-1-\sigma(i)\}!}
=
\det
\begin{bmatrix} 
\dfrac{1}
{(k_i+d-1-j)!}
\end{bmatrix} _{0 \le i,j \le d-1}
\\
&=
\frac{(-1)^{d(d-1)/2} \Delta(k_i)}
{\{k_i+d-1\}!}.
\end{align*}
}%
The last equality follows from the lemma below.
\end{proof}

\begin{lemma}
[{\cite[Example A.9.3]{fulton}}]
\label{lemma:det}
$$
\det 
\begin{bmatrix} 
\dfrac{1}{(x_i + j)! } 
\end{bmatrix} _{0 \le i,j \le d-1}
=
\frac{\Delta(x_i)}{\pr{x_i + d-1} }
.
$$
\end{lemma}

\begin{proof}[Proof of Theorem \ref{theorem:main_theorem}]
By Proposition \ref{prop:Laurent_series} 
and  
Lemma \ref{lemma:constant_part_linear_form}
\eqref{lemma:constant_part_linear_form_1} 
with $A:= A^{*}(X)\otimes \mathbb Q$, 
we have 
{\setlength{\multlinegap}{36pt}
\begin{multline}
\label{last computation}
\pi_*
\ch(\det \mathcal Q) 
=
\Phi
\Big(
\prod_{i=0}^{d-1}
t_i^{-(d-1-i)}
\prod_{i=0}^{d-1} \big( t_i^{r -d}s(\mathcal E, t_i) \big)
\Big)
\\=
(-1)^{d(d-1)/2}
\Phi
\Big(
\prod_{i=0}^{d-1}
t_i^{-i}
\prod_{i=0}^{d-1}\big( t_i^{r-d}s(\mathcal E, t_i) \big)
\Big)
.
\end{multline}
}%
Since 
$$
\prod_{i=0}^{d-1}s(\mathcal E, t_i)
=
\sum_{k}
\prod_{i=0}^{d-1} 
s_{k_i}(\mathcal E) 
t_i^{k_i}
, 
$$
it follows from 
Proposition 
\ref{prop:evaluation_formula} 
that 
the most right-hand side of (\ref{last computation}) is equal to
$$
(-1)^{d(d-1)/2}
\sum_k 
\Phi
\Big(
\prod_{i=0}^{d-1}
t_i^{r -d+k_i-i}
\prod_{i=0}^{d-1} 
s_{k_i}(\mathcal E)
\Big)
=
\sum_k 
\frac{\Delta(k_{i}-i)}{
\{r +k_{i}-i-1\}!}
\prod_{i=0}^{d-1}s_{k_i}(\mathcal E)  ,
$$
where 
$k = (k_0, \dots , k_{d-1}) \in \mathbb Z_{\ge 0}^d$. 
Thus we obtain the conclusion. 
\end{proof}

\begin{proof}[Proof of Corollary \ref{corollary:degree_formula}]
By the assumption  
$\grassmann$  
is projective and 
the tautological line bundle $\mathcal O_{\mathbb P_X(\wedge^{d} \mathcal E)}(1)$ 
defines 
an embedding 
$\mathbb P_X(\wedge^{d} \mathcal E) \hookrightarrow 
\mathbb P(H^0(X, \wedge^{d} \mathcal E))$. 
Therefore 
$\grassmann$
is considered to be a projective variety 
in $\mathbb P(H^0(X, \wedge^{d} \mathcal E))$ 
via 
the relative Pl\"ucker embedding 
$\grassmann \hookrightarrow \mathbb P_X(\wedge^{d} \mathcal E)$
over $X$ 
defined by the quotient 
$\wedge^d\pi^*\mathcal E \to \wedge^d \mathcal Q=\det \mathcal Q$. 
Since the hyperplane section class of $\grassmann$ is equal to the Pl\"ucker class $\theta$, 
we obtain the conclusion, 
taking the degree of the equality 
in Theorem \ref{theorem:main_theorem}. 
\end{proof}

\begin{proof}[Proof of Theorem \ref{theorem:another_formula}]
By Lemmas 
\ref{lemma:Cauchy_formula} below, 
\ref{lemma:constant_part_linear_form} \eqref{lemma:constant_part_linear_form_3} 
and Proposition \ref{prop:evaluation_formula}, 
the most right-hand side of (\ref{last computation}) is equal to
{\allowdisplaybreaks %
\begin{align*}
(-1)^{d(d-1)/2}
&
\sum_{\lambda}
\Phi
\Big(
\prod_{i=0}^{d-1}t_i^{r-d-i}
s_{\lambda}(\underline{t}) 
\Big)
\varDelta_{\lambda}(s(\mathcal E))
\\
&=
(-1)^{d(d-1)/2}
\sum_{\lambda}
\Phi
\Big(
\prod_{i=0}^{d-1}t_i^{r-d-i+\lambda_{i+1}}
\Big)
\varDelta_{\lambda}(s(\mathcal E))
\\
&=
 \sum_{\lambda}
\frac{\Delta(r-d-i+\lambda_{i+1})}
{\{  r-d-i+\lambda_{i+1} + (d-1)  \}!}
\varDelta_{\lambda}(s(\mathcal E))
\\
&=
\sum_{\lambda}
\frac{\Delta(\lambda_{i+1}-(i+1))}{
\{\lambda_{i+1}+r-(i+1)\}!}
\varDelta_{\lambda}(s(\mathcal E)) 
=
\sum_{\lambda}
\frac{f^{\lambda + \varepsilon}}{\vert \lambda + \varepsilon \vert !}
\varDelta_{\lambda}(s(\mathcal E))
. 
\qedhere
\end{align*}
}%
\end{proof}

\begin{lemma}
\label{lemma:Cauchy_formula}
$$
\prod_{i=0}^{d-1} s(\mathcal E,t_i)
=
\sum_{\lambda}
\varDelta_{\lambda}(s(\mathcal E))
s_{\lambda}(\underline{t}) 
. 
$$
\end{lemma}

\begin{proof}
Using Cauchy identity \cite[Chapter I, (4.3)]{macdonald}
and 
Jacobi-Trudi identity 
\cite[Lemma A.9.3]{fulton}, we have
$$
\prod_{i=0}^{d-1} s(\mathcal E,t_i)=
\prod_{i=0}^{d-1}\frac{1}{c(\mathcal E,-t_i)}
=\prod_{i=0}^{d-1}\prod_{j=1}^r
\frac{1}{1-\alpha_jt_i} 
=
\sum_{\lambda}s_{\lambda}(\underline{\alpha})s_{\lambda}(\underline{t}) 
=
\sum_{\lambda}
\varDelta_{\lambda}(s(\mathcal E))
s_{\lambda}(\underline{t}) 
,
$$
where 
$\underline \alpha = \{ \alpha_1, \dots , \alpha_{r} \}$ 
are the Chern roots of the vector bundle $\mathcal E$. 
\end{proof}

\section{Appendix: A generalization of {Cauchy Determinant Formula}}

Consider a polynomial ring 
$R_1:=A[\xi_0, \dots, \xi_{r-1}]$ with 
$r$ variables over a $\mathbb Q$-algebra $A$. 
Denote by $c''_i$ 
the $i$-th elementary symmetric polynomial in $\xi_d,\dots, \xi_{r-1}$, 
and by $c_i$
the $i$-th elementary symmetric polynomial in $\xi_0,\dots, \xi_{r-1}$. 
We define the Segre series $s(t)$ by
$$
s(t):=\frac{1}{\prod_{i=0}^{r-1}(1-\xi_it)}.
$$
Set 
$R_2:=A[\xi_0, \dots, \xi_{d-1},c''_1, \dots, c''_{r-d}]$, 
and 
$R_3:=A[c_1, \dots, c_{r}]$.
Then, 
$R_1 \supset R_2 \supset R_3$, and 
$R_1$ (resp. $R_2)$ is a free $R_3$-modules generated by 
$\{\xi_0^{i_0}\cdots \xi_{r-1}^{i_{r-1}}\}$ 
(resp. $\{\xi_0^{i_0}\cdots \xi_{d-1}^{i_{d-1}}\}$), 
where $0\leq i_l \leq r-l-1$ 
(see, {\it e.g.}, \cite[Chapitre 4, \S6]{bourbaki}, \cite[\S\S2--3]{laksov}). 
In particular, we have a decomposition, 
\begin{equation}\label{equation:cohomology_ring_genral_setting}
R_2 
= 
\bigoplus_{\substack{0\le i_l \le {r} -l-1\\(0\le l \le {d} -1)}}
R_3 
\cdot 
\xi_0^{i_0} 
\xi_1^{i_1} \cdots 
\xi_{{d}-1}^{i_{{d} -1}} . 
\end{equation}
For $\alpha \in R_2$, we denote by 
$\coeff_{\underline{\xi}}(\alpha)$
the coefficient of $\alpha$ in 
$\xi_0^{r-1}\cdots \xi_{d-1}^{r-d}$
with respect to the decomposition 
\eqref{equation:cohomology_ring_genral_setting}.

Let ${\mathcal A}$ (resp. ${\mathcal A}'$, ${\mathcal A}''$) 
be the anti-symmetrizer for variables $\{ \xi_0, \dots,$ $\xi_{r-1} \}$
(resp. $\{ \xi_0,\dots, \xi_{d-1} \}$, 
$\{ \xi_{d} , \dots,$ $\xi_{r-1} \}$), 
that is, 
${\mathcal A}(\alpha):=\sum_{\sigma \in \mathfrak S_r} \sgn(\sigma)\sigma(\alpha)$ 
$(\alpha \in R_1)$, for instance.

\begin{proposition}[Generalization of {Cauchy Determinant Formula}]
\label{prop:generalization_of_Cauchy_identity}
We have an equality
$$
{\mathcal A}
\Big(
\frac{\Delta(\xi_0, \dots,\xi_{d-1})\Delta(\xi_d,\dots,\xi_{r-1})}
{\prod_{0\leq i,j\leq d-1}(\tau_j-\xi_i)}
\Big)
=
\frac{\Delta(\xi_0, \dots,\xi_{r-1})}
{\prod_{0\leq i\leq r-1,0\leq j\leq d-1}(\tau_j-\xi_i)}
.
$$
By setting $\tau_i:=\dfrac{1}{t_i}$, we have
$$
{\mathcal A}
\Big(
\frac{\Delta(\xi_0,\dots, \xi_{d-1})
\cdot \Delta(\xi_d, \dots,\xi_{r-1})}{\prod_{0\leq i,j \leq d-1}(1-\xi_it_j)}
\Big)=\frac{\Delta(\xi_0,\dots, \xi_{r-1})\prod_{i=0}^{d-1}t_i^{r-d}}
{\prod_{0\leq i \leq r-1,
0 \leq j\leq d-1}(1-\xi_it_j)}
.
$$
\end{proposition}

\begin{proof}
The fractional expression, 
$$
{\mathcal A}
\Big(
\frac{\Delta(\xi_0, \dots,\xi_{d-1})\Delta(\xi_d,  \dots,\xi_{r-1})}
{\prod_{0\leq i,j\leq d-1}(\tau_j-\xi_i)}
\Big)
\prod_{0\leq i\leq r-1,0\leq j \leq d-1}(\tau_j-\xi_i)
$$
is actually 
a homogeneous polynomial 
in the variables,  
$\xi_0, \dots, \xi_{r-1}$, 
$\tau_0, \dots, \tau_{r-1}$, with degree 
${d(d-1)}/{2}+{(r-d)(r-d-1)}/{2}-d^2+rd={r(r-1)}/{2}$, 
and anti-symmetric 
with respect to the $\xi_i$. 
Therefore it is a multiple of $\Delta(\xi_0, \dots, \xi_{r-1})$.
By comparing the coefficient of $\xi_0^{r-1}\cdots \xi_{r-1}^0$,
we see that those polynomials are equal to each other, and 
we obtain the first equality. 
The second equality follows from the first one.
 \end{proof}

\begin{proof}[Another Proof of Lemma \ref{lemma:monomial}] 
Let $G(\underline t)$ be the generating function of 
$\coeff_{\underline{\xi}}
(\xi_0^{p_0}\cdots \xi_{d-1}^{p_{d-1}}),
$ 
that is, 
$$
G(\underline t) 
:= \sum_{p_0, \dots, p_{d-1}\geq 0}
\coeff_{\underline{\xi}}(\xi_{0}^{p_{0}} \cdots \xi_{d-1}^{p_{d-1}} )
t_0^{p_0} \cdots t_{d-1}^{p_{d-1}} 
.
$$ 
For $0\leq i_l \leq r-l-1$, we have
$$
{\mathcal A}(\xi_0^{i_0}\cdots \xi_{r-1}^{i_{r-1}})=
\begin{cases}
\Delta(\xi_0, \dots, \xi_{r-1}) , & (i_0,\dots, i_{r-1})= (r-1, \dots, 0) , 
\\
0 , & (i_0,\dots, i_{r-1})\neq (r-1, \dots, 0). \\
\end{cases}
$$
Since ${\mathcal A}$ is $R_3$-linear, we have an equality，
$$
{\mathcal A}(\alpha\cdot\xi_d^{r-d-1}\cdots \xi_{r-1}^0)
=
\coeff_{\underline{\xi}}
(\alpha)
\Delta(\xi_0,\dots, \xi_{r-1})
$$
in $R_1$ for $\alpha\in R_2$.
Therefore,  
{\allowdisplaybreaks %
\begin{align*}
\begin{split}
{\Delta(\xi_0,\dots, \xi_{r-1})}
{G(\underline t)}
&=
\sum_{p_0, \dots, p_{d-1}\geq 0}{\mathcal A}(\xi_0^{p_0},\dots, \xi_{d-1}^{p_{d-1}}
\cdot\xi_d^{r-d-1}\cdots \xi_{r-1}^0)\ t_0^{p_0}\cdots t_{d-1}^{p_{d-1}} \\
&=
{\mathcal A}
\Big(
\frac{\xi_d^{r-d-1}\cdots \xi_{r-1}^0}
{(1-\xi_0t_0)\cdots (1-\xi_{d-1}t_{d-1})}
\Big)
\\
&=
{\mathcal A}
\Big(
{\mathcal A}'
\Big(
\frac{1}{(1-\xi_0t_0)\cdots (1-\xi_{d-1}t_{d-1})}
\Big)
{\mathcal A}''(\xi_d^{r-d-1}\cdots \xi_{r-1}^0)
\Big)
\\
&=
{\mathcal A}
\Big(
\frac{
\Delta(t_0,\dots, t_{d-1})
\Delta(\xi_0,\dots, \xi_{d-1})
\Delta(\xi_d, \dots,\xi_{r-1})}{\prod_{0\leq i,j \leq d-1}(1-\xi_it_j)}
\Big) 
\\
&=
{\mathcal A}
\Big(
\frac{
\Delta(\xi_0,\dots, \xi_{d-1})
\Delta(\xi_d, \dots,\xi_{r-1})}{\prod_{0\leq i,j \leq d-1}(1-\xi_it_j)}
\Big)
\Delta(t_0,\dots, t_{d-1})
. 
\end{split}
\end{align*}
}%
Here we used the equality, 
$$
{\mathcal A}(f(\xi_0, \dots, \xi_{d-1})
g(\xi_d,\dots, \xi_{r-1}))
=
{\mathcal A}({\mathcal A}'(f(\xi_0, \dots, \xi_{d-1}))
{\mathcal A}''(g(\xi_d,\dots, \xi_{r-1})))
$$
and Cauchy determinant formula (\cite[p.67, I.4, Example 6]{macdonald}). 
Finally, using Proposition \ref{prop:generalization_of_Cauchy_identity}, 
we see that 
$$
{G(\underline t)}
=
\frac{\Delta(t_0,\dots, t_{d-1})\prod_{i=0}^{d-1}t_i^{r-d}}
{\prod_{0\leq i \leq r-1,
0 \leq j\leq d-1}(1-\xi_it_j)} 
=
\Delta(t_0,\dots, t_{d-1})\prod_{i=0}^{d-1}t_i^{r-d} s(t_i)
,
$$
and this proves Lemma \ref{lemma:monomial} with $R_{1}:= A^{*}(X)$ and 
$R_{2}:= A^{*}(\X{d})= A^{*}(\G{d-1})$. 
\end{proof} 

\smallskip 

\noindent%
{\it Acknowlgments.} 
The authors thank Professor Hiroshi Naruse and 
Professor Takeshi Ikeda, too, for useful discussion and kind advice. 
The first author is supported by JSPS KAKENHI Grant Number 25400053. 
The second author is supported by JSPS KAKENHI Grant Number 
15H02048.

\end{document}